\documentclass[10pt]{amsart}

  \baselineskip 1 mm

\usepackage{amsmath,amsthm,amsfonts,latexsym,amsopn,verbatim,amscd,amssymb}
\usepackage{hyperref}

\theoremstyle{plain}
\newtheorem{theorem}{Theorem}[section]
\newtheorem{lemma}[theorem]{Lemma}

\newtheorem{corollary}[theorem]{Corollary}
\theoremstyle{definition}

\numberwithin{equation}{section}

\newcommand{\bnum}{\begin{enumerate}}
\newcommand{\enum}{\end{enumerate}}

\begin{document}

\title[Genus of commuting conjugacy class graph of  groups]{Genus of commuting conjugacy class graph of finite groups}
\author[P. Bhowal and R. K. Nath]{Parthajit Bhowal  and Rajat Kanti Nath$^*$}
\address{Department of Mathematical Sciences, Tezpur
University,  Napaam-784028, Sonitpur, Assam, India.
}
\email{bhowal.parthajit8@gmail.com, rajatkantinath@yahoo.com (corresponding author).}

\subjclass[2010]{20D99, 05C25, 05C10,   05E15.}
\keywords{commuting conjugacy class graph, genus, finite group.}

%

\thanks{*Corresponding author}

%
%
\begin{abstract}
For a non-abelian group $G$, its commuting conjugacy class graph $\mathcal{CCC}(G)$  is a simple undirected graph whose vertex set is the set of conjugacy classes of the non-central elements of $G$ and two distinct vertices $x^G$ and $y^G$ are adjacent if there exists some elements $x' \in x^G$ and $y' \in y^G$ such that $x'y' = y'x'$. 
In this paper we compute the genus of $\mathcal{CCC}(G)$ for six  well-known classes of non-abelian two-generated groups (viz. $D_{2n}, SD_{8n}, Q_{4m}, V_{8n}, U_{(n, m)}$ and $G(p, m, n)$) and determine whether $\mathcal{CCC}(G)$ for these groups are planar, toroidal, double-toroidal or triple-toroidal.   
\end{abstract}

\maketitle

%
%

\section{Introduction} \label{S:intro}

The commuting conjugacy class graph of a non-abelian group $G$ is a simple undirected  graph, denoted by $\mathcal{CCC}(G)$, whose vertex set is the set of conjugacy classes of the non-central elements of $G$ and two distinct vertices $x^G$ and $y^G$ are adjacent if there exists some elements $x' \in x^G$ and $y' \in y^G$ such that $x'y' = y'x'$.  This graph extends the notion of commuting graph of a finite group introduced by Brauer and Fowler \cite{bF1955}, in 1955. Commuting graphs of finite algebraic structures, its extensions, generalizations and their complements remain active topic of research over the years. In 2009, Herzog, Longobardi and Maj \cite{hLM2009} initiated the study of  commuting conjugacy class graph of groups. In 2016, finite groups having triangle-free commuting conjugacy class graph were characterized by Mohammadian et al. \cite{mefw}.   
Ashrafi and Salahshour  have also considered commuting conjugacy class graph of finite groups  in their recent work \cite{sA2020}, where they obtain  structures of $\mathcal{CCC}(G)$ for the following groups:
  
\begin{align*}
D_{2n} &= \langle \alpha, \beta : \alpha^n = \beta^2 = 1, \beta\alpha \beta = \alpha^{-1}\rangle \text{ for } n \geq 3,\\
SD_{8n} &= \langle \alpha, \beta : \alpha^{4n} = \beta^2 = 1, \beta\alpha \beta = \alpha^{2n -1}\rangle \text{ for } n \geq 2, \\
Q_{4m} &= \langle \alpha, \beta : \alpha^{2m} = 1, \alpha^m = \beta^2, \beta^{-1}\alpha \beta = \alpha^{-1}\rangle \text{ for } m \geq 2,\\
V_{8n} &= \langle \alpha, \beta : \alpha^{2n} = \beta^4 = 1, \beta\alpha = \alpha^{-1}\beta^{-1}, \beta^{-1}\alpha = \alpha^{-1}\beta \rangle \text{ for } n \geq 2,\\
U_{(n, m)} &= \langle \alpha, \beta :  \alpha^{2n} = \beta^m = 1, \alpha^{-1}\beta\alpha = \beta^{-1}\rangle \text{ for } m \geq 2 \text{ and } n \geq 2 \text{ and }\\
G(p, m, n) &= \langle \alpha, \beta : \alpha^{p^m} = \beta^{p^n} = [\alpha, \beta]^p = 1, [\alpha, [\alpha, \beta]] = [\beta, [\alpha, \beta]] = 1\rangle, 
\end{align*}
where $p$ is any prime,  $m \geq 1$ and $n \geq 1$.

Continuing the works of Ashrafi and Salahshour \cite{sA2020}, in \cite{bN-CCG-1-20,bN-CCG-2-20} Bhowal and Nath have obtained various spectra and energies of commuting conjugacy class graphs of finite groups.  
In this paper we compute genus of commuting conjugacy class graph of the above  mentioned groups and determine whether $\mathcal{CCC}(G)$ for those groups are planar, toroidal, double-toroidal or triple-toroidal. 
The genus $\gamma(\mathcal{G})$ of a graph $\mathcal{G}$ is the smallest  integer $k \geq 0$ such that $\mathcal{G}$ can be embedded on the surface obtained by attaching $k$ handles to a sphere. A graph $\mathcal{G}$ is called planar, toroidal, double-toroidal or triple-toroidal if $\mathcal{G}$ has genus $0,1,2$ or $3$ respectively. Results on genus of commuting graphs of finite groups, including  its various extensions,  can be found in \cite{AF14,bnN20,das13,das2015}. However, genus of commuting conjugacy class graph of finite groups are not yet studied. 

\section{Genus of $\mathcal{CCC}(G)$ and characterizations}
Let $K_n$ be the complete graph on $n$ vertices and $mK_n$ the disjoint union of $m$ copies of $K_n$. Then, by \cite[Theorem 6-38]{whit}, we have
\begin{equation} \label{kn}
\gamma(K_{n})=\left\lceil \frac{(n-3)(n-4)}{12}\right
\rceil, \text{ if } n \geq 3.
\end{equation}
By \cite[Corollary 2]{bhky}, we also have the following lemma.
\begin{lemma}\label{kmn}
If $\mathcal{G} = m_1K_{n_1} \sqcup m_2K_{n_2}$ then $\gamma(\mathcal{G}) = m_1\gamma(K_{n_1}) + m_2\gamma(K_{n_2})$.
\end{lemma}
Now we compute genus of commuting conjugacy class graph of the groups $D_{2n}, SD_{8n}, Q_{4m}, V_{8n}$, $U_{(n, m)}$ and $G(p, m, n)$ one by one and check their planarity, toroidality etc.
\begin{theorem}
Let $G = D_{2n}$. Then
\begin{enumerate}
\item $\mathcal{CCC}(G)$ is planar if and only if $3\leq n \leq 10$.
\item $\mathcal{CCC}(G)$ is toroidal if and only if $11\leq n \leq 16$.
\item $\mathcal{CCC}(G)$ is double-toroidal if and only if $n = 17,18$.
\item $\mathcal{CCC}(G)$ is triple-toroidal if and only if $n = 19,20$.
\item $\gamma(\mathcal{CCC}(G))=\begin{cases}
\left\lceil\frac{(n-7)(n-9)}{48}\right\rceil, & \text{ if $n$ is odd and $n\geq 21$ } 
\vspace{.2cm}\\
\left\lceil\frac{(n - 8)(n - 10)}{48}\right\rceil, & \text{ if $n$ is even and $n\geq 22$. }
\end{cases}$
\end{enumerate}
\end{theorem}
\begin{proof}
Consider the following cases.

\noindent \textbf{Case 1.} $n$ is odd.

 By \cite[Proposition 2.1]{sA2020} we have $\mathcal{CCC}(G) = K_1 \sqcup K_{\frac{n - 1}{2}}$. Therefore, for $n = 3$ and $5$, it follows that $\mathcal{CCC}(G)= 2K_1$, $K_1\sqcup K_2$ respectively; and hence $\mathcal{CCC}(G)$ is planar. If $n \geq 7$ then, by Lemma \ref{kmn} and \eqref{kn}, we have 
 \[
 \gamma(\mathcal{CCC}(G)) = \gamma(K_{\frac{n - 1}{2}}) =  \left\lceil\frac{(n-7)(n-9)}{48}\right\rceil.
 \]
Clearly $\gamma(\mathcal{CCC}(G)) = 0$ if and only if $n=7$ or $9$. Also, $\gamma(\mathcal{CCC}(G)) = 1$ if $n = 11$, $13$ or $15$; $\gamma(\mathcal{CCC}(G)) = 2$ if $n = 17$; $\gamma(\mathcal{CCC}(G)) = 3$ if $n = 19$. For $n\geq 21$ we have 
\[
\frac{(n-7)(n-9)}{48} \geq \frac{7}{2} = 3.5,
\]
and so 
 \[
 \gamma(\mathcal{CCC}(G)) =  \left\lceil\frac{(n-7)(n-9)}{48}\right\rceil \geq 4.
 \]
 Thus, $\mathcal{CCC}(G)$ is planar if and only if $n=3,5,7,9$; toroidal if and only if $n=11,13,15$; double-toroidal if and only if $n=17$ and triple-toroidal if and only if $n=19$.
 
 \noindent \textbf{Case 2.} $n$ is even.

 By \cite[Proposition 2.1]{sA2020} we have
 \[
 \mathcal{CCC}(G) =  \begin{cases}
 2K_1 \sqcup K_{\frac{n}{2} - 1}, & \text{ if $n$ and $\frac{n}{2}$ are even}\\
  K_2 \sqcup K_{\frac{n}{2} - 1}, & \text{ if $n$ is even and $\frac{n}{2}$ is odd.}
\end{cases}
 \] 
 Therefore, for $n = 4$ and $6$, it follows that $\mathcal{CCC}(G) = 3K_1$, $2K_2$ respectively; and hence $\mathcal{CCC}(G)$ is planar. If $n\geq 8$ then, by Lemma \ref{kmn} and \eqref{kn}, we have 
  \[
 \gamma(\mathcal{CCC}(G)) = \gamma(K_{\frac{n}{2} - 1}) =  \left\lceil\frac{(n - 8)(n - 10)}{48}\right\rceil.
 \]
 Clearly $\gamma(\mathcal{CCC}(G)) = 0$ if and only if $n=8$ or $10$. Also, $\gamma(\mathcal{CCC}(G)) = 1$ if $n = 12$, $14$ or $16$; $\gamma(\mathcal{CCC}(G)) = 2$ if $n = 18$; $\gamma(\mathcal{CCC}(G)) = 3$ if $n = 20$. For $n\geq 22$ we have 
\[
\frac{(n-8)(n-10)}{48} \geq \frac{7}{2} = 3.5,
\]
and so 
 \[
 \gamma(\mathcal{CCC}(G)) =  \left\lceil\frac{(n-8)(n-10)}{48}\right\rceil \geq 4.
 \]
 Thus, $\mathcal{CCC}(G)$ is planar if and only if $n=4,6,8,10$; toroidal if and only if $n=12,14,16$; double-toroidal if and only if $n=18$ and triple-toroidal if and only if $n=20$. Hence the result follows.
 \end{proof}
\begin{theorem} 
Let $G=SD_{8n}$. Then
\begin{enumerate}
\item $\mathcal{CCC}(G)$ is planar if and only if $n= 2$ or $3$.
\item $\mathcal{CCC}(G)$ is toroidal if and only if $n=4$.
\item $\mathcal{CCC}(G)$ is double-toroidal if and only if $n=5$.
\item $\mathcal{CCC}(G)$ is not triple-toroidal.
\item $\gamma(\mathcal{CCC}(G))=\begin{cases}
\left\lceil\frac{(n - 3)(2n - 5)}{6}\right\rceil, & \text{ if $n$ is odd and $n\geq 7$ }
\vspace{.2cm}\\
\left\lceil\frac{(n - 2)(2n - 5)}{6}\right\rceil, & \text{ if $n$ is even and $n\geq 6$. }
\end{cases}$
\end{enumerate}
\end{theorem}
\begin{proof}
Consider the following cases.

\noindent \textbf{Case 1.} $n$ is odd.

By \cite[Proposition 2.1]{sA2020} we have $\mathcal{CCC}(G) = K_4 \sqcup K_{2n - 2}$. For $n \geq 3$, by Lemma \ref{kmn} and \eqref{kn}, we have 
 \[
 \gamma(\mathcal{CCC}(G)) = \gamma(K_{4}) + \gamma(K_{2n - 2}) =  \left\lceil\frac{(n-3)(2n-5)}{6}\right\rceil.
 \]
Clearly $\gamma(\mathcal{CCC}(G)) = 0$ if $n = 3$; $\gamma(\mathcal{CCC}(G)) = 2$ if $n = 5$. For $n\geq 7$ we have 
\[
\frac{(n-3)(2n-5)}{6} \geq 6,
\]
and so 
 \[
 \gamma(\mathcal{CCC}(G)) =  \left\lceil\frac{(n-3)(2n-5)}{6}\right\rceil \geq 6.
 \]
Thus $\mathcal{CCC}(G)$ is planar if and only if $n=3$; double-toroidal if and only if $n=5$.
 
 \noindent \textbf{Case 2.} $n$ is even.

By \cite[Proposition 2.1]{sA2020} we have $\mathcal{CCC}(G) = 2K_1 \sqcup K_{2n - 1}$. For $n \geq 2$, by Lemma \ref{kmn} and \eqref{kn}, we have 
 \[
 \gamma(\mathcal{CCC}(G)) = \gamma(K_{2n - 1}) =  \left\lceil\frac{(n-2)(2n-5)}{6}\right\rceil.
 \]
Clearly $\gamma(\mathcal{CCC}(G)) = 0$ if $n = 2$; $\gamma(\mathcal{CCC}(G)) = 1$ if $n = 4$. For $n\geq 6$ we have 
\[
\frac{(n-2)(2n-5)}{6} \geq \frac{14}{3},
\]
and so 
 \[
  \gamma(\mathcal{CCC}(G)) = \left\lceil\frac{(n-2)(2n-5)}{6}\right\rceil \geq 5.
 \]
Thus $\mathcal{CCC}(G)$ is planar if and only if $n=2$; toroidal if and only if $n=4$. Hence the result follows. 
 \end{proof}
 
\begin{theorem}
Let $G= Q_{4m}$. Then
\begin{enumerate}
\item $\mathcal{CCC}(G)$ is planar if and only if $m= 2,3,4$ or $5$.
\item $\mathcal{CCC}(G)$ is toroidal if and only if $m= 6,7$ or $8$.
\item $\mathcal{CCC}(G)$ is double-toroidal if and only if $m=9$.
\item $\mathcal{CCC}(G)$ is triple-toroidal if and only if $m=10$.
\item $\gamma(\mathcal{CCC}(G)) =  \left\lceil\frac{(m - 4)(m - 5)}{12}\right\rceil$ for $m\geq 11$.
\end{enumerate}
\end{theorem}
\begin{proof}
 By \cite[Proposition 2.1]{sA2020} we have
 \[
 \mathcal{CCC}(G) =  \begin{cases}
 K_2 \sqcup K_{m - 1}, & \text{ if $m$ is odd}\\
  2K_1 \sqcup K_{m - 1}, & \text{ if $m$ is even.}
\end{cases}
 \] 
 Therefore, for $m = 2, 3$, it follows that $\mathcal{CCC}(G) = 3K_1$, $2K_2$ respectively; and hence $\mathcal{CCC}(G)$ is planar. If $m\geq 4$ then, by Lemma \ref{kmn} and \eqref{kn}, we have 
  \[
 \gamma(\mathcal{CCC}(G)) = \gamma(K_{m - 1}) =  \left\lceil\frac{(m - 4)(m - 5)}{12}\right\rceil.
 \]
 Clearly $\gamma(\mathcal{CCC}(G)) = 0$ if and only if $m=4$ or $5$. Also, $\gamma(\mathcal{CCC}(G)) = 1$ if $m = 6$, $7$ or $8$; $\gamma(\mathcal{CCC}(G)) = 2$ if $m = 9$; $\gamma(\mathcal{CCC}(G)) = 3$ if $m = 10$. For $m\geq 11$ we have 
\[
\frac{(m - 4)(m - 5)}{12} \geq \frac{7}{2} = 3.5,
\]
and so 
 \[
 \gamma(\mathcal{CCC}(G)) =  \left\lceil\frac{(m - 4)(m - 5)}{12}\right\rceil \geq 4.
 \]
 Thus, $\mathcal{CCC}(G)$ is planar if and only if $m = 2,3,4,5$; toroidal if and only if $m = 6,7,8$; double-toroidal if and only if $m=9$ and triple-toroidal if and only if $m=10$. Hence the result follows.
\end{proof}
\begin{theorem}
Let $G=V_{8n}$. Then
\begin{enumerate}
\item $\mathcal{CCC}(G)$ is planar if and only if $n= 2$.
\item $\mathcal{CCC}(G)$ is toroidal if and only if $n=3$ or $4$.
\item $\mathcal{CCC}(G)$ is not double-toroidal.
\item $\mathcal{CCC}(G)$ is triple-toroidal if and only if $n= 5$.
\item $\gamma(\mathcal{CCC}(G))=\begin{cases}
\left\lceil\frac{(n - 2)(2n - 5)}{6}\right\rceil, & \text{ if $n$ is odd and $n\geq 7$ }
\vspace{.2cm}\\
\left\lceil\frac{(n - 3)(2n - 5)}{6}\right\rceil, & \text{ if $n$ is even and $n\geq 6$. }
\end{cases}$
\end{enumerate}
\end{theorem}
\begin{proof}
Consider the following cases.

\noindent \textbf{Case 1.} $n$ is odd.

By \cite[Proposition 2.1]{sA2020} we have $\mathcal{CCC}(G) = 2K_1 \sqcup K_{2n - 1}$. For $n \geq 3$, by Lemma \ref{kmn} and \eqref{kn}, we have 
 \[
 \gamma(\mathcal{CCC}(G)) = \gamma(K_{2n - 1}) =  \left\lceil\frac{(n-2)(2n-5)}{6}\right\rceil.
 \]
Clearly $\gamma(\mathcal{CCC}(G)) = 1$ if $n = 3$; $\gamma(\mathcal{CCC}(G)) = 3$ if $n = 5$. For $n\geq 7$ we have 
\[
\frac{(n-2)(2n-5)}{6} \geq \frac{15}{2} = 7.5,
\]
and so 
 \[
 \gamma(\mathcal{CCC}(G)) =  \left\lceil\frac{(n-2)(2n-5)}{6}\right\rceil \geq 8.
 \]

 \noindent \textbf{Case 2.} $n$ is even.

By \cite[Proposition 2.1]{sA2020} we have $\mathcal{CCC}(G) = 2K_2 \sqcup K_{2n - 2}$. Therefore, for $n = 2$ it follows that $\mathcal{CCC}(G) = 3K_2$; and hence $\mathcal{CCC}(G)$ is planar. If $n \geq 4$ then, by Lemma \ref{kmn} and \eqref{kn}, we have 
 \[
 \gamma(\mathcal{CCC}(G)) = \gamma(K_{2n - 2}) =  \left\lceil\frac{(n-3)(2n-5)}{6}\right\rceil.
 \]
Clearly $\gamma(\mathcal{CCC}(G)) = 1$ if $n = 4$. For $n\geq 6$ we have 
\[
\frac{(n-3)(2n-5)}{6} \geq \frac{7}{2} = 3.5,
\]
and so 
 \[
  \gamma(\mathcal{CCC}(G)) = \left\lceil\frac{(n-3)(2n-5)}{6}\right\rceil \geq 4.
 \]
 Thus $\mathcal{CCC}(G)$ is planar if and only if $n=2$; toroidal if and only if $n=4$. Hence the result follows.
\end{proof}

\begin{theorem}
Let $G= U_{(n,m)}$. Then
\begin{enumerate}
\item $\mathcal{CCC}(G)$ is planar if and only if $n = 2$ and $m = 2, 3, 4, 5, 6$; $n = 3$ and $m = 2, 3, 4$; or $n = 4$ and $m = 2, 3, 4$.
\item $\mathcal{CCC}(G)$ is toroidal if and only if $n = 2$ and  $m = 7, 8$; or $n = 3$ and $m = 5, 6$.
\item $\mathcal{CCC}(G)$ is double-toroidal if and only if $n = 2$ and $m = 9, 10$; $n = 4$ and $m = 5, 6$; $n = 5$ and $m = 2, 3$; $n = 6$ and $m = 2, 3$; or $n = 7$ and $m = 2, 3$.
\item $\mathcal{CCC}(G)$ is triple-toroidal if and only if $n = 3$ and $m = 7, 8$; $n = 5$ and $m = 4$; $n = 6$ and $m = 4$; or $n = 7$ and $m = 4$.
\item $\gamma(\mathcal{CCC}(G)) = \begin{cases}
\left\lceil\frac{(mn - n - 6)(mn - n - 8)}{48}\right\rceil, & \text{if $n=2$, $m$ is odd and $m \geq 11$} 
\vspace{.2cm}\\
\left\lceil\frac{(mn - 2n - 6)(mn - 2n - 8)}{48}\right\rceil, & \text{if $n=2$, $m$ is even and $m \geq 12$}
\vspace{.2cm}\\
\left\lceil\frac{(mn - n - 6)(mn - n - 8)}{48}\right\rceil + \left\lceil\frac{(n - 3)(n - 4)}{12}\right\rceil, & \text{if $n = 3$, $m$ is odd and $m\geq 9$};\\
& n = 4, m\geq 7; n = 5, m\geq 5;\\
& n = 6,m\geq 5;  n = 7, m\geq 5;\\
& \text{ or } n \geq 8, m\geq 3 \\
\left\lceil\frac{(mn - 2n - 6)(mn - 2n - 8)}{48}\right\rceil + 2  \left\lceil\frac{(n - 3)(n - 4)}{12}\right\rceil, &\text{if $n = 3$,  $m$ is even and $m\geq 10$};\\
& n = 4, m\geq 8; n = 5, m\geq 6;\\
& n = 6,m\geq 6;  n = 7, m\geq 6;\\
&\text{ or } n \geq 8, m\geq 2
\end{cases}$
\end{enumerate}
\end{theorem}
\begin{proof}
Consider the following cases.

\noindent \textbf{Case 1.} $m$ is odd.

By \cite[Proposition 2.3]{sA2020} we have $\mathcal{CCC}(G) =   K_{\frac{n(m - 1)}{2}} \sqcup K_n$.

\noindent \textbf{Sub case 1.1} $n=2.$

 If $n=2$ then we have $\mathcal{CCC}(G) =   K_{m - 1} \sqcup K_2$. Therefore, for $m = 3$ it follows that $\mathcal{CCC}(G) = 2K_2$; and hence $\mathcal{CCC}(G)$ is planar. For $m\geq 5$, by Lemma \ref{kmn}, we have 
 \[
 \gamma(\mathcal{CCC}(G)) = \gamma(K_{m - 1}) = \left\lceil\frac{(m - 4)(m - 5)}{12}\right\rceil.
 \]
Clearly $\gamma(\mathcal{CCC}(G)) = 0$ if $m = 5$; $\gamma(\mathcal{CCC}(G)) = 1$ if $m = 7$; $\gamma(\mathcal{CCC}(G)) = 2$ if $m = 9$. For $m\geq 11$ we have 
\[
\frac{(m - 4)(m - 5)}{12} \geq \frac{7}{2} = 3.5,
\]
and so 
 \[
  \gamma(\mathcal{CCC}(G)) = \left\lceil\frac{(m - 4)(m - 5)}{12}\right\rceil \geq 4.
 \]
Thus $\mathcal{CCC}(G)$ is planar if and only if $m = 3,5$; toroidal if and only if $m=7$; double-toroidal if and only if $m=9$.

\noindent \textbf{Sub case 1.2} $n\geq 3.$

 If $n\geq 3$ then we have $\mathcal{CCC}(G) =   K_{\frac{n(m - 1)}{2}} \sqcup K_n$. By Lemma \ref{kmn}, we have 
 \[
 \gamma(\mathcal{CCC}(G)) = \gamma(K_{\frac{n(m - 1)}{2}}) + \gamma(K_n) = \left\lceil\frac{(mn - n - 6)(mn - n - 8)}{48}\right\rceil + \left\lceil\frac{(n - 3)(n - 4)}{12}\right\rceil.
 \]
 Clearly $\gamma(\mathcal{CCC}(G)) = 0$ if $n=3, m=3$ or $n=4, m=3$. $\gamma(\mathcal{CCC}(G)) = 1$ if $n=3, m=5$; $\gamma(\mathcal{CCC}(G)) = 2$ if $n=4, m=5$ or $n=5, m=3$ or $n=6, m=3$ or $n=7, m=3$; $\gamma(\mathcal{CCC}(G)) = 3$ if $n=3,m= 7$. If $n = 3$ and $m\geq 9$ then
 \[
\frac{(mn - n - 6)(mn - n - 8)}{48} = \frac{(m - 3)(3m - 11)}{16} \geq 6.
\]
Therefore
\[
\gamma(\mathcal{CCC}(G)) = \left\lceil\frac{(mn - n - 6)(mn - n - 8)}{48}\right\rceil + \left\lceil\frac{(n - 3)(n - 4)}{12}\right\rceil \geq 6.
\]
If $n = 4$ and $m \geq 7$ then
\[
\frac{(mn - n - 6)(mn - n - 8)}{48} = \frac{(2m - 5)(m - 3)}{6} \geq 6.
\]
Therefore
 \[
\gamma(\mathcal{CCC}(G)) = \left\lceil\frac{(mn - n - 6)(mn - n - 8)}{48}\right\rceil + \left\lceil\frac{(n - 3)(n - 4)}{12}\right\rceil \geq 6.
\]
If $n = 5$ and $m \geq 5$ then
\[
\frac{(mn - n - 6)(mn - n - 8)}{48} = \frac{(5m - 11)(5m - 13)}{48} \geq \frac{7}{2} = 3.5 \text{ \, and \,} \frac{(n - 3)(n - 4)}{12} = \frac{1}{6}.
\]
Therefore
 \[
\gamma(\mathcal{CCC}(G)) = \left\lceil\frac{(mn - n - 6)(mn - n - 8)}{48}\right\rceil + \left\lceil\frac{(n - 3)(n - 4)}{12}\right\rceil \geq 5.
\]
If $n = 6$ and $m \geq 5$ then
\[
\frac{(mn - n - 6)(mn - n - 8)}{48} = \frac{(m - 2)(3m - 7)}{4} \geq 6 \text{ \, and \,} \frac{(n - 3)(n - 4)}{12} = \frac{1}{2}.
\]
Therefore
 \[
\gamma(\mathcal{CCC}(G)) = \left\lceil\frac{(mn - n - 6)(mn - n - 8)}{48}\right\rceil + \left\lceil\frac{(n - 3)(n - 4)}{12}\right\rceil \geq 7.
\]
If $n = 7$ and $m \geq 5$ then
\[
\frac{(mn - n - 6)(mn - n - 8)}{48} = \frac{(7m - 13)(7m - 15)}{48} \geq \frac{55}{6} \text{ \, and \,} \frac{(n - 3)(n - 4)}{12} = 1.
\]
Therefore
 \[
\gamma(\mathcal{CCC}(G)) = \left\lceil\frac{(mn - n - 6)(mn - n - 8)}{48}\right\rceil + \left\lceil\frac{(n - 3)(n - 4)}{12}\right\rceil \geq 11.
\]
If $n \geq 8$ and $m \geq 3$ then
\[
\frac{(mn - n - 6)(mn - n - 8)}{48} \geq \frac{(8(m - 1) - 6)(8(m - 1) - 7)}{48} \geq \frac{15}{8} \text{ \, and \,} \frac{(n - 3)(n - 4)}{12} = \frac{5}{3}.
\]
Therefore
 \[
\gamma(\mathcal{CCC}(G)) = \left\lceil\frac{(mn - n - 6)(mn - n - 8)}{48}\right\rceil + \left\lceil\frac{(n - 3)(n - 4)}{12}\right\rceil \geq 4.
\]
Thus $\mathcal{CCC}(G)$ is planar if and only if $n=3, m=3$ or $n=4, m=3$; toroidal if and only if $n=3, m=5$; double-toroidal if and only if $n=4, m=5$ or $n=5, m=3$ or $n=6, m=3$ or $n=7, m=3$; triple-toroidal if and only if $n=3,m= 7$.

\noindent \textbf{Case 2.} $m$ is even.

By \cite[Proposition 2.3]{sA2020} we have $\mathcal{CCC}(G) =   K_{\frac{n(m - 2)}{2}} \sqcup 2K_n$.

\noindent \textbf{Sub case 2.1} $n=2.$

 If $n=2$ then we have $\mathcal{CCC}(G) =   K_{m - 2} \sqcup 2K_2$. Therefore, for $m = 2, 4$ it follows that $\mathcal{CCC}(G) = 2K_2$ and $3K_2$; and hence $\mathcal{CCC}(G)$ is planar. For $m\geq 6$, by Lemma \ref{kmn}, we have 
 \[
 \gamma(\mathcal{CCC}(G)) = \gamma(K_{m - 2}) = \left\lceil\frac{(m - 5)(m - 6)}{12}\right\rceil.
 \]
Clearly $\gamma(\mathcal{CCC}(G)) = 0$ if $m = 6$; $\gamma(\mathcal{CCC}(G)) = 1$ if $m = 8$; $\gamma(\mathcal{CCC}(G)) = 2$ if $m = 10$. For $m\geq 12$ we have 
\[
\frac{(m - 5)(m - 6)}{12} \geq \frac{7}{2} = 3.5
\]
and so 
 \[
  \gamma(\mathcal{CCC}(G)) = \left\lceil\frac{(m - 4)(m - 5)}{12}\right\rceil \geq 4.
 \]
Thus $\mathcal{CCC}(G)$ is planar if and only if $m = 2,4,6$; toroidal if and only if $m=8$; double-toroidal if and only if $m=10$.

\noindent \textbf{Sub case 2.2} $n\geq 3.$

 If $n\geq 3$ then we have $\mathcal{CCC}(G) = K_{\frac{n(m - 2)}{2}} \sqcup 2K_n$. By Lemma \ref{kmn}, we have 
 \[
 \gamma(\mathcal{CCC}(G)) = \gamma(K_{\frac{n(m - 2)}{2}}) + \gamma(2K_n) = \left\lceil\frac{(mn - 2n - 6)(mn - 2n - 8)}{48}\right\rceil + 2 \left\lceil\frac{(n - 3)(n - 4)}{12}\right\rceil.
 \]
 Clearly $\gamma(\mathcal{CCC}(G)) = 0$ if $n=3, m=2,4$ or $n=4, m=2,4$. $\gamma(\mathcal{CCC}(G)) = 1$ if $n=3, m=6$; $\gamma(\mathcal{CCC}(G)) = 2$ if $n=4, m=6$ or $n=5, m=2$ or $n=6, m=2$ or $n=7, m=2$; $\gamma(\mathcal{CCC}(G)) = 3$ if $n=3,m = 8$ or $n = 5, m = 4$ or $n = 6, m = 4$ or $n = 7, m = 4$. If $n = 3$ and $m\geq 10$ then
 \[
\frac{(mn - 2n - 6)(mn - 2n - 8)}{48} = \frac{(m - 4)(3m - 14)}{16} \geq 6.
\]
Therefore
\[
\gamma(\mathcal{CCC}(G)) = \left\lceil\frac{(mn - 2n - 6)(mn - 2n - 8)}{48}\right\rceil + 2 \left\lceil\frac{(n - 3)(n - 4)}{12}\right\rceil \geq 6.
\]
If $n = 4$ and $m \geq 8$ then
 \[
\frac{(mn - 2n - 6)(mn - 2n - 8)}{48} = \frac{(m - 4)(2m - 7)}{6} \geq 6.
\]
Therefore
\[
\gamma(\mathcal{CCC}(G)) = \left\lceil\frac{(mn - 2n - 6)(mn - 2n - 8)}{48}\right\rceil + 2 \left\lceil\frac{(n - 3)(n - 4)}{12}\right\rceil \geq 6.
\]
If $n = 5$ and $m \geq 6$ then
 \[
\frac{(mn - 2n - 6)(mn - 2n - 8)}{48} = \frac{(5m - 16)(5m - 18)}{48} \geq \frac{7}{2} = 3.5 \text{ \, and \,} \frac{(n - 3)(n - 4)}{12} = \frac{1}{6}.
\]
Therefore
\[
\gamma(\mathcal{CCC}(G)) = \left\lceil\frac{(mn - 2n - 6)(mn - 2n - 8)}{48}\right\rceil + 2 \left\lceil\frac{(n - 3)(n - 4)}{12}\right\rceil \geq 6.
\]
If $n = 6$ and $m \geq 6$ then
 \[
\frac{(mn - 2n - 6)(mn - 2n - 8)}{48} = \frac{(m - 3)(3m - 10)}{4} \geq 6 \text{ \, and \,} \frac{(n - 3)(n - 4)}{12} = \frac{1}{6}.
\]
Therefore
\[
\gamma(\mathcal{CCC}(G)) = \left\lceil\frac{(mn - 2n - 6)(mn - 2n - 8)}{48}\right\rceil + 2 \left\lceil\frac{(n - 3)(n - 4)}{12}\right\rceil \geq 8.
\]
If $n = 7$ and $m \geq 6$ then
 \[
\frac{(mn - 2n - 6)(mn - 2n - 8)}{48} = \frac{(7m - 20)(7m - 22)}{48} \geq \frac{55}{6} \text{ \, and \,} \frac{(n - 3)(n - 4)}{12} = 1.
\]
Therefore
\[
\gamma(\mathcal{CCC}(G)) = \left\lceil\frac{(mn - 2n - 6)(mn - 2n - 8)}{48}\right\rceil + 2  \left\lceil\frac{(n - 3)(n - 4)}{12}\right\rceil \geq 12.
\]
If $n = 8$ and $m \geq 2$ then
\[
\frac{(n - 3)(n - 4)}{12} = \frac{5}{3} \text{ \, and \,} \left\lceil\frac{(mn - 2n - 6)(mn - 2n - 8)}{48}\right\rceil \geq 0.
\]
Therefore
\[
\gamma(\mathcal{CCC}(G)) = \left\lceil\frac{(mn - 2n - 6)(mn - 2n - 8)}{48}\right\rceil + 2  \left\lceil\frac{(n - 3)(n - 4)}{12}\right\rceil \geq 4.
\]
Thus $\mathcal{CCC}(G)$ is planar if and only if $n=3, m=2,4$ or $n=4, m=2,4$; toroidal if and only if $n=3, m=6$; double-toroidal if and only if $n=4, m=6$ or $n=5, m=2$ or $n=6, m=2$ or $n=7, m=2$; triple-toroidal if and only if $n=3,m = 8$ or $n = 5, m = 4$ or $n = 6, m = 4$ or $n = 7, m = 4$. Hence the result follows.
\end{proof}
\begin{theorem}
Let $G=G(p,m,n)$. Then
\begin{enumerate}
\item $\mathcal{CCC}(G)$ is planar if and only if $n=1, m=1, p= 2,3,5$; $n=1, m=2, p=2$; $n=1, m=3, p=2$; $n=2, m=1, p=2$; $n=2, m=2, p=2$; or $n=3, m=1, p=2$.
\item $\mathcal{CCC}(G)$ is not toroidal.
\item $\mathcal{CCC}(G)$ is double-toroidal if and only if $n=2, m=1, p=3$.
\item $\mathcal{CCC}(G)$ is not triple-toroidal.
\item $\gamma(\mathcal{CCC}(G))= \begin{cases}
(p + 1)  \left\lceil\frac{(p - 4)(p - 5)}{12}\right\rceil, & \text{if $n=1, m=1, p\geq 7$}\vspace{.2cm}\\

\vspace{.2cm}

(p + 1)  \left\lceil\frac{(p^2 - p - 3)(p^2 - p - 4)}{12}\right\rceil, & \text{if $n=1, m=2, p\geq 3$}\\

\vspace{.2cm}

(p + 1)  \left\lceil\frac{(p^3 - p^2 - 3)(p^3 - p^2 - 4)}{12}\right\rceil, & \text{if $n=1, m=3, p\geq 3$}\\
(p + 1)  \left\lceil\frac{(p^m - p^{m - 1} - 3)(p^m - p^{m - 1} - 4)}{12}\right\rceil, & \text{if $n=1, m \geq 3, p\geq 2$}\vspace{.2cm}\\
(p^2 - p)  \left\lceil\frac{(p - 4)(p - 5)}{12}\right\rceil + 2  \left\lceil\frac{(p^2 - p - 3)(p^2 - p - 4)}{12}\right\rceil, & \text{if $n=2, m=1, p\geq 5$} \vspace{.4cm}\\

(p^2 - 2)  \left\lceil\frac{(p^2 - p - 3)(p^2 - p - 4)}{12}\right\rceil \\
\quad \quad \quad + 2  \left\lceil\frac{(p^3 - p^2 - 3)(p^3 - p^2 - 4)}{12}\right\rceil, & \text{if $n=2, m=2, p\geq 3$}\vspace{.4cm}\\

(p^2 - 2)  \left\lceil\frac{(p^{m - 1}(p - 1) - 3)(p^{m - 1}(p - 1) - 4)}{12}\right\rceil\\
\quad \quad \quad + 2  \left\lceil\frac{(p^m(p - 1) - 3)(p^m(p - 1) - 4)}{12}\right\rceil, & \text{if $n=2, m\geq 3, p\geq 2$}\vspace{.4cm}\\

p^2(p - 1)  \left\lceil\frac{(p - 4)(p - 5)}{12}\right\rceil + 2  \left\lceil\frac{(p^3 - p^2 - 3)(p^3 - p^2 - 4)}{12}\right\rceil, & \text{if $n=3, m=1, p\geq 3$}\vspace{.2cm}\\

4, & \text{if $n=3, m=2, p=2$}\vspace{.4cm}\\

p^2(p - 1)  \left\lceil\frac{(p^2 - p - 3)(p^2 - p - 4)}{12}\right\rceil\\
\quad \quad \quad + 2  \left\lceil\frac{(p^4 - p^3 - 3)(p^4 - p^3 - 4)}{12}\right\rceil, & \text{if $n=3, m\geq 2, p\geq 3$} \vspace{.4cm}\\

2  \left\lceil\frac{(p^{n - 1}(p^m - p^{m - 1}) - 3)(p^{n - 1}(p^m - p^{m - 1}) - 4)}{12}\right\rceil, & \text{if $n\geq 4, m\geq 1, p\geq 2$}\\
&\text{ and $p^m - p^{m - 1} \leq 2$}\\
(p^n - p^{n - 1})   \left\lceil\frac{(p^m - p^{m - 1} - 3)(p^m - p^{m - 1} - 4)}{12}\right\rceil \\ 
\quad \quad \quad + 2  \left\lceil\frac{(p^{n - 1}(p^m - p^{m - 1}) - 3)(p^{n - 1}(p^m - p^{m - 1}) - 4)}{12}\right\rceil, & \text{if $n\geq 4, m\geq 1, p\geq 2$}\\
&\text{ and $p^m - p^{m - 1} \geq 3$}.\\
\end{cases}$
\end{enumerate}
\end{theorem}
\begin{proof}

By \cite[Proposition 2.6]{sA2020} we have 
\[
\mathcal{CCC}(G) = (p^n - p^{n - 1})K_{p^{m - n}(p^n - p^{n - 1})} \sqcup K_{p^{n - 1}(p^m - p^{m - 1})} \sqcup K_{p^{m - 1}(p^n - p^{n - 1})}.
\]

Consider the following cases.

\noindent \textbf{Case 1.} $n = 1$.

We have $\mathcal{CCC}(G) = (p + 1)K_{p^{m - 1}(p - 1)}$. For $m = 1$ and $p = 2, 3$, it follows that $\mathcal{CCC}(G) = 2K_1$ or $3K_2$ which is planar. If $m = 1$ and $p\geq 5$, by Lemma \ref{kmn} and \eqref{kn}, we have 
\[
 \gamma(\mathcal{CCC}(G)) = (p + 1)\gamma(K_{p - 1}) =  (p + 1)  \left\lceil\frac{(p - 4)(p - 5)}{12}\right\rceil.
\]
Clearly $\gamma(\mathcal{CCC}(G)) = 0$ for $p = 5$. If $p \geq 7$ then
\[
\frac{(p - 4)(p - 5)}{12} \geq \frac{1}{2}
\]
and so
\[
\gamma(\mathcal{CCC}(G)) = (p + 1)  \left\lceil\frac{(p - 4)(p - 5)}{12}\right\rceil \geq 8.
\]
If $m = 2$ and $p = 2$ then $\gamma(\mathcal{CCC}(G)) = 3\gamma(K_{2}) = 0$. For $m = 2$ and $p\geq 3$, by Lemma \ref{kmn} and \eqref{kn}, we have
\[
\gamma(\mathcal{CCC}(G)) = (p+1) \gamma(K_{p(p - 1)}) = (p+1)  \left\lceil\frac{(p^2 - p - 3)(p^2 - p - 4)}{12}\right\rceil.
\]
If $p \geq 3$ then 
\[
\frac{(p^2 - p - 3)(p^2 - p - 4)}{12} \geq \frac{1}{2}
\]
and so 
\[
\gamma(\mathcal{CCC}(G)) = (p + 1)  \left\lceil\frac{(p^2 - p - 3)(p^2 - p - 4)}{12}\right\rceil \geq 4.
\]
If $m = 3$ then $\gamma(\mathcal{CCC}(G)) = (p+1) \gamma(K_{p^2(p - 1)})$. Therefore, if $m = 3$ and $p\geq 2$ then by Lemma \ref{kmn} and \eqref{kn}, we have
\[
\gamma(\mathcal{CCC}(G)) = (p+1) \gamma(K_{p^2(p - 1)}) = (p+1)  \left\lceil\frac{(p^3 - p^2 - 3)(p^3 - p^2 - 4)}{12}\right\rceil.
\]
Clearly if $m = 3$ and $p = 2$ then $\gamma(\mathcal{CCC}(G)) = 0$. If $p \geq 3$ then 
\[
\frac{(p^3 - p^2 - 3)(p^3 - p^2 - 4)}{12} \geq \frac{35}{2}
\]
and so
\[
\gamma(\mathcal{CCC}(G)) = (p + 1)  \left\lceil\frac{(p^3 - p^2 - 3)(p^3 - p^2 - 4)}{12}\right\rceil \geq 72.
\]
If $m \geq 4$ and $p \geq 2$ then $\gamma(\mathcal{CCC}(G)) = (p+1) \gamma(K_{p^{m - 1}(p - 1)})$. Therefore,  by Lemma \ref{kmn} and \eqref{kn}, we have
\[
\gamma(\mathcal{CCC}(G)) = (p+1) \gamma(K_{p^{m - 1}(p - 1)}) = (p+1)  \left\lceil\frac{(p^m - p^{m - 1} - 3)(p^m - p^{m - 1} - 4)}{12}\right\rceil.
\]
We have 
\[
\frac{(p^m - p^{m - 1} - 3)(p^m - p^{m - 1} - 4)}{12} \geq \frac{20}{12}
\]
and so
\[
\gamma(\mathcal{CCC}(G)) = (p+1)  \left\lceil\frac{(p^m - p^{m - 1} - 3)(p^m - p^{m - 1} - 4)}{12}\right\rceil \geq 6.
\]
Therefore, $\mathcal{CCC}(G)$ is planar if and only if  $n=1, m=1, p= 2,3,5$; $n=1, m=2, p=2$; or $n=1, m=3, p=2$. Also, in this case, $\mathcal{CCC}(G)$ is neither toroidal, double-toridal nor triple-toroidal.

\newpage
\noindent \textbf{Case 2.} $n = 2$.

We have $\mathcal{CCC}(G) = (p^2 - p)K_{p^{m - 1}(p - 1)} \sqcup 2K_{p^m(p - 1)}$. For $m = 1$ and $p = 2$, it follows that $\mathcal{CCC}(G) = 2K_1 \sqcup 2K_2$ which is planar.  If $m = 1$ and $p = 3$ then, by Lemma \ref{kmn} and \eqref{kn}, we have 
\[
 \gamma(\mathcal{CCC}(G)) = 2\gamma(K_6) = 2.
\]
If $m = 1$ and $p\geq 5$, by Lemma \ref{kmn} and \eqref{kn}, we have 
\begin{align*}
 \gamma(\mathcal{CCC}(G)) &= (p^2 - p)\gamma(K_{p - 1}) + 2\gamma(K_{p(p - 1)}) \\
 & = (p^2 - p)  \left\lceil\frac{(p - 4)(p - 5)}{12}\right\rceil + 2  \left\lceil\frac{(p^2 - p - 3)(p^2 - p - 4)}{12}\right\rceil.
\end{align*}
If $p \geq 5$ then
\[
\frac{(p^2 - p - 3)(p^2 - p - 4)}{12} \geq \frac{68}{3}
\]
and so
\[
\gamma(\mathcal{CCC}(G)) \geq 2  \left\lceil\frac{(p^2 - p - 3)(p^2 - p - 4)}{12}\right\rceil \geq 46.
\]
If $m = 2$ and $p \geq 2$ then $\mathcal{CCC}(G) = (p^2 - p)K_{p(p - 1)} \sqcup 2K_{p^2(p - 1)}$. Therefore, if $p = 2$ then $\mathcal{CCC}(G) = 2K_2 \sqcup 2K_4$ hence by \eqref{kn} we have $\gamma(\mathcal{CCC}(G)) = 2\gamma(K_4) = 0$. If $p\geq 3$, by Lemma \ref{kmn} and \eqref{kn}, we have
\begin{align*}
\gamma(\mathcal{CCC}(G)) &= (p^2 - p) \gamma(K_{p(p - 1)}) + 2 \gamma(K_{p^2(p - 1)}) \\
&= (p^2 - 2)  \left\lceil\frac{(p^2 - p - 3)(p^2 - p - 4)}{12}\right\rceil + 2  \left\lceil\frac{(p^3 - p^2 - 3)(p^3 - p^2 - 4)}{12}\right\rceil.
\end{align*}
Also,
$
\frac{(p^3 - p^2 - 3)(p^3 - p^2 - 4)}{12} \geq \frac{35}{2}
$
and so 
\[
\gamma(\mathcal{CCC}(G)) \geq
2  \left\lceil\frac{(p^3 - p^2 - 3)(p^3 - p^2 - 4)}{12}\right\rceil \geq 36.
\]
If $m \geq 3$ and $p\geq 2$ then
\begin{align*}
 \gamma(\mathcal{CCC}&(G)) = (p^2 - p)\gamma(K_{p^{m - 1}(p - 1)}) + 2\gamma(K_{p^m(p - 1)}) \\
 & = (p^2 - p)  \left\lceil\frac{(p^{m - 1}(p - 1) - 3)(p^{m - 1}(p - 1) - 4)}{12}\right\rceil + 2  \left\lceil\frac{(p^m(p - 1) - 3)(p^m(p - 1) - 4)}{12}\right\rceil.
\end{align*}
Therefore, $\mathcal{CCC}(G)$ is planar if and only if $n=2, m=1, p=2$; $n=2, m=2, p=2$; or $n=3; m=1; p=2$ and double-toroidal if and only if $n=2, m=1, p=3$. In this case, $\mathcal{CCC}(G)$ is neither toroidal nor triple-toroidal. 

\noindent \textbf{Case 3.} $n = 3$.

We have $\mathcal{CCC}(G) = p^2(p-1)K_{p^{m - 1}(p- 1)} \sqcup 2K_{p^{m+1}(p-1)}$. If $m = 1$ and $p = 2$ then $\mathcal{CCC}(G) = 4K_1 \sqcup 2K_4$, and so by \eqref{kn} $\gamma(\mathcal{CCC}(G)) = 2\gamma(K_4) = 0$. For $p = 3$ we have $\mathcal{CCC}(G) = 18K_2 \sqcup 2K_{18}$. Therefore,  by \ref{kmn} and \eqref{kn} we have $\gamma(\mathcal{CCC}(G)) = 2\gamma(K_{18}) = 36$. For $m = 1$ and $p \geq 3$, by Lemma \ref{kmn} and \eqref{kn} we have
\begin{align*}
\gamma(\mathcal{CCC}(G)) &= p^2(p - 1) \gamma(K_{p - 1}) + 2 \gamma(K_{p^2(p - 1)}) \\
&= p^2(p - 1)  \left\lceil\frac{(p - 4)(p - 5)}{12}\right\rceil + 2  \left\lceil\frac{(p^3 - p^2 - 3)(p^3 - p^2 - 4)}{12}\right\rceil.
\end{align*}
If $p\geq 5$ then
\[
\frac{(p^3 - p^2 - 3)(p^3 - p^2 - 4)}{12} = 776,
\]
and so 
\[
\gamma(\mathcal{CCC}(G)) \geq 2  \left\lceil\frac{(p^3 - p^2 - 3)(p^3 - p^2 - 4)}{12}\right\rceil \geq 1552.
\]
If $m = 2$ and $p = 2$ then we have $\mathcal{CCC}(G) = 4K_2 \sqcup 2K_8$. By Lemma \ref{kmn} and \eqref{kn} we have
\[
\gamma(\mathcal{CCC}(G)) = 2 \gamma(K_8) = 4.
\]
If $m = 2$ and $p\geq 3$ then we have $\mathcal{CCC}(G) = p^2(p-1)K_{p(p- 1)} \sqcup 2K_{p^3(p-1)}$. By Lemma \ref{kmn} and \eqref{kn} we have
\begin{align*}
\gamma(\mathcal{CCC}(G)) &= p^2(p - 1) \gamma(K_{p(p - 1)}) + 2 \gamma(K_{p^3(p-1)}) \\
&= p^2(p - 1)  \left\lceil\frac{(p^2 - p - 3)(p^2 - p - 4)}{12}\right\rceil + 2  \left\lceil\frac{(p^4 - p^3 - 3)(p^4 - p^3 - 4)}{12}\right\rceil.
\end{align*}
If $p \geq 3$ then
\[
\frac{(p^2 - p - 3)(p^2 - p - 4)}{12} \geq \frac{1}{2}
\]
and so
\[
\gamma(\mathcal{CCC}(G)) > p^2(p - 1)  \left\lceil\frac{(p^2 - p - 3)(p^2 - p - 4)}{12}\right\rceil \geq 18.
\]
If $m \geq 3$ and $p\geq 2$ then we have $\mathcal{CCC}(G) = p^2(p-1)K_{p^{m - 1}(p- 1)} \sqcup 2K_{p^{m+1}(p-1)}$. By Lemma \ref{kmn} and \eqref{kn} we have
\begin{align*}
\gamma(\mathcal{CCC}(G)) &= p^2(p - 1) \gamma(K_{p^{m - 1}(p- 1)}) + 2 \gamma(K_{p^{m+1}(p-1)}) \\
&= p^2(p - 1)  \left\lceil\frac{(p^{m - 1}(p- 1) - 3)(p^{m - 1}(p- 1) - 4)}{12}\right\rceil \\
& \quad \quad \quad \quad \quad \quad \quad + 2  \left\lceil\frac{(p^{m+1}(p-1) - 3)(p^{m+1}(p-1) - 4)}{12}\right\rceil.
\end{align*}
We have
\[
\frac{(p^{m+1}(p-1) - 3)(p^{m+1}(p-1) - 4)}{12} \geq 13
\]
and so
\[
\gamma(\mathcal{CCC}(G)) \geq 2  \left\lceil\frac{(p^{m+1}(p-1) - 3)(p^{m+1}(p-1) - 4)}{12}\right\rceil \geq 26.
\]
Therefore, $\mathcal{CCC}(G)$ is planar if and only if $n=3, m=1, p=2$. Also, in this case, $\mathcal{CCC}(G)$ is neither toroidal, double-toridal nor triple-toroidal.

\noindent \textbf{Case 4.} $n \geq 4$.

We have 
\[
\mathcal{CCC}(G) = (p^n - p^{n - 1})K_{p^m - p^{m - 1}} \sqcup 2K_{p^{n - 1}(p^m - p^{m - 1})}.
\]
Therefore, by Lemma \ref{kmn}, we have
\begin{equation}\label{1}
\gamma(\mathcal{CCC}(G)) = (p^n - p^{n - 1})\gamma(K_{p^m - p^{m - 1}}) + 2\gamma(K_{p^{n - 1}(p^m - p^{m - 1})})
\end{equation}
For $m\geq 1$ and $p\geq 2$ we have 
\[
\gamma(K_{p^{n - 1}(p^m - p^{m - 1})}) \geq \gamma(K_{p^{n - 1}}) \geq \gamma(K_8) = 2,
\] 
noting that $K_8$ and $K_{p^{n - 1}}$ are subgraphs of $K_{p^{n - 1}}$ and $K_{p^{n - 1}(p^m - p^{m - 1})}$ respectively. Therefore
\begin{align*}
\gamma(\mathcal{CCC}(G)) \geq 2\gamma(K_{p^{n - 1}(p^m - p^{m - 1})}) \geq 4.
\end{align*}
Further, if $p^m - p^{m - 1} \leq 2$ then, by \eqref{1} and \eqref{kn}, we have
\begin{align*}
\gamma(\mathcal{CCC}(G)) &= 2\gamma(K_{p^{n - 1}(p^m - p^{m - 1})})\\
& = 2  \left\lceil\frac{(p^{n - 1}(p^m - p^{m - 1}) - 3)(p^{n - 1}(p^m - p^{m - 1}) - 4)}{12}\right\rceil.
\end{align*}
If $p^m - p^{m - 1} \geq 3$ then, by \eqref{1} and \eqref{kn}, we have
\begin{align*}
\gamma(\mathcal{CCC}(G)) 
& = (p^n - p^{n - 1})   \left\lceil\frac{(p^m - p^{m - 1} - 3)(p^m - p^{m - 1} - 4)}{12}\right\rceil +\\
& \quad \quad \quad \quad 2  \left\lceil\frac{(p^{n - 1}(p^m - p^{m - 1}) - 3)(p^{n - 1}(p^m - p^{m - 1}) - 4)}{12}\right\rceil.
\end{align*}
Hence the result follows.
\end{proof}
We conclude this paper with the following characterization of $\mathcal{CCC}(G)$.
\begin{corollary}
Let $G = D_{2n}, SD_{8n}, Q_{4m}, V_{8n}$, $U_{(n, m)}$ or $G(p, m, n)$. Then
\begin{enumerate}
\item $\mathcal{CCC}(G)$ is planar if and only if $G = D_{6}, D_{8}, D_{10}, D_{12}, D_{14}, D_{16}, D_{18}, D_{20}, SD_{16}, SD_{24}, Q_{8}$, $Q_{12}, Q_{16}, Q_{20}, V_{16}, U_{(2, 2)}, U_{(2, 3)}, U_{(2, 4)}, U_{(2, 5)}, U_{(2, 6)}, U_{(3, 2)}, U_{(3, 3)}, U_{(3, 4)}, U_{(4, 2)}, U_{(4, 3)}$, \, $U_{(4, 4)}$,  $G(2, 1, 1), G(3, 1, 1), G(5, 1, 1)$,
$G(2, 2, 1), G(2, 3, 1), G(2, 1, 2), G(2, 2, 2)$ or $G(2, 1, 3)$.

\item $\mathcal{CCC}(G)$ is toroidal if and only if $G = D_{22}, D_{24}, D_{26}, D_{28}, D_{30}, D_{32}, SD_{32}, Q_{24}, Q_{28}, Q_{32}, V_{24}$, $V_{32}, U_{(2, 7)}, U_{(2, 8)},  U_{(3, 5)}$ or $U_{(3, 6)}$.
\item $\mathcal{CCC}(G)$ is double-toroidal if and only if $G = D_{34}, D_{36}, SD_{40}, Q_{36}, U_{(2, 9)}, U_{(2, 10)}, U_{(4, 5)}, U_{(4, 6)}$, $U_{(5, 2)}, U_{(5, 3)},  U_{(6, 2)}, U_{(6, 3)}, U_{(7, 2)}, U_{(7, 3)}$ or $G(3, 1, 2)$.
\item $\mathcal{CCC}(G)$ is triple-toroidal if and only if $G = D_{38}, D_{40}, Q_{40}, V_{40}, U_{(3, 7)}, U_{(3, 8)}, U_{(5, 4)}, U_{(6, 4)}$ or $U_{(7, 4)}$.
\end{enumerate}
\end{corollary}

\noindent {\bf Acknowledgements}
    The first author is thankful to Council of Scientific and Industrial Research  for the fellowship (File No. 09/796(0094)/2019-EMR-I).

\end{document}